\documentclass[12pt]
{amsart}
\usepackage{amssymb,amsmath,amsthm,amsfonts,amsopn,url,color,
mathtools,microtype,MnSymbol,dutchcal,
mathrsfs, changepage}
\usepackage[colorlinks=true, pagebackref=true]%
{hyperref}
\usepackage{scalerel}
\usepackage{stackengine,wasysym}
\usepackage[shortlabels]{enumitem}

\usepackage[normalem]{ulem}
\usepackage[all]{xy}
\input xy
\xyoption{all}
\usepackage{amscd}
\usepackage{soul}

\theoremstyle{plain}
\newtheorem{thm}{Theorem}[section]

\newtheorem{prop}[thm]{Proposition}
\newtheorem{lemma}[thm]{Lemma}
\newtheorem{cor}[thm]{Corollary}

\theoremstyle{definition}
\newtheorem{defn}[thm]{Definition}
\newtheorem*{defn*}{Definition}
\newtheorem*{question*}{Question}
\newtheorem{question}{Question}
\newtheorem{example}[thm]{Example}
\newtheorem*{example*}{Example}
\newtheorem{rem}[thm]{Remark}
\newtheorem*{rem*}{Remark}

\newcommand{\field}[1]{\mathbb{#1}}
\newcommand{\N}{\field{N}}
\newcommand{\Z}{\field{Z}}

\newcommand{\R}{\field{R}}
\newcommand{\F}{\field{F}}
\newcommand{\C}{\field{C}}
\newcommand{\A}{\field{A}}
\newcommand{\G}{\field{G}}

\newcommand{\ideal}[1]{\mathfrak{#1}}
\newcommand{\m}{\ideal{m}}

\newcommand{\p}{\ideal{p}}
\newcommand{\q}{\ideal{q}}
\newcommand{\ia}{\ideal{a}}

\newcommand{\func}[1]{\mathrm{#1} \,}

\newcommand{\Spec}{\func{Spec}}
\DeclareMathOperator{\Min}{Min}

\newcommand{\arrow}[1]{\stackrel{#1}{\rightarrow}}

\newcommand{\ra}{\rightarrow}

\newcommand{\be}{\begin{enumerate}}
\newcommand{\ee}{\end{enumerate}}

\newcommand{\li}
 {\leftfootline}

\newcommand{\onto}{\twoheadrightarrow}

\newcommand{\into}{\hookrightarrow}

\newcommand{\cA}{\mathcal{A}}
\newcommand{\cB}{\mathcal{B}}
\newcommand{\cC}{\mathcal{C}}

\newcommand{\cO}{\mathcal{O}}

\renewcommand{\phi}{\varphi}

\DeclareMathOperator{\chr}{char}
\DeclareMathOperator{\Max}{Max}

\newcommand{\red}{{\mathrm{red}}}

\let\int\relax
\DeclareMathOperator{\int}{i}

\newcommand{\ua}{unit-additive}

\newcommand{\alg}{{\rm{alg}}}

\newcommand{\pins}{purely inseparable}

\begin{document}

\title[Fundamental varieties and unit-additive rings]
{Fundamental algebraic sets and locally unit-additive rings}

\author{Neil Epstein}
\address{Department of Mathematical Sciences \\ George Mason University \\ Fairfax, VA  22030}
\email{nepstei2@gmu.edu}

\date{October 16, 2025}

\begin{abstract}
The Fundamental Theorem of Algebra can be thought of as a statement about the real numbers as a space, considered as an algebraic set over the real numbers as a field.  This paper introduces what it means for an algebraic set or affine  variety over a field to be \emph{fundamental}, in a way that encompasses the Fundamental Theorem of Algebra as a special case. The related concept 
of \emph{local} 
fundamentality is introduced and its behavior developed.  On the algebraic side, the notions of \emph{locally}, \emph{geometrically}, and \emph{generically unit-additive} 
rings are introduced, thus complementing \emph{unit-additivity} as previously defined by the author and Jay Shapiro. A number of results are extended from the previous joint paper from unit-additivity to local unit-additivity. It is shown that an affine variety is (locally) fundamental if and only if its coordinate ring is (locally) unit-additive.  To do so, a theorem is proved showing that there are many equivalent definitions of local unit-additivity. Illustrative examples are sprinkled throughout.
\end{abstract}
\maketitle

\section{Introduction}

One way to state the Fundamental Theorem of Algebra is this.  Given a regular function $g$ that is nonconstant on the real line, there is some point $P$ in the complexification of the real line (i.e. a complex number) such that $g(P) =0$.  One does have to 
adjoin the square root of $-1$
in general; otherwise $g=X^2+4$ would be a counterexample.

In 
an introductory algebraic geometry course (see, e.g. \cite{Hart-AG, Per-AG}), one studies more general \emph{affine algebraic sets}.  That is, given a field $k$, one takes a collection $f_1, \ldots, f_s \in k[X_1, \ldots, X_n]$ of polynomials and studies both its vanishing set $V=V_k(f_1, \ldots, f_s)$ over $k$ and $\bar V = V_{k^\alg}(f_1, \ldots, f_s)$ over the algebraic closure $k^\alg$ of $k$.  One may then ask the same question: given a regular function $g$ that is not constant on $V$, is there some $P \in \bar V$ such that $g(P) = 0$?  If so, we will say that $V$ is \emph{fundamental} (over $k$), in the sense that it \emph{satisfies the fundamental theorem of algebra} over that field, in the extended sense above.  Note that the above is independent of the embedding of $V$ as an affine variety (see Proposition~\ref{pr;embedindependent}).

Throughout the dual histories of the subjects of commutative algebra and algebraic geometry, a common theme has been to find two apparently unrelated properties in the two subjects and show that they have a close correspondence. It is shown in this paper that such a correspondence exists between the above notion of a \emph{fundamental algebraic set} and that of a \emph{unit-additive ring} (see Theorem~\ref{thm:FTAiffua}).

Recall from \cite{nmeSh-unitadd}, where the notion was first investigated, that a ring $R$ is \emph{\ua} if a sum of units is always a unit or nilpotent. This notion has its roots in the much more restrictive notion of UU ring \cite{Cal-UU}, where every unit is congruent to 1 mod nilpotents, or even in the study of rings where $1$ is the only unit \cite{HeiRo-utriv, DJ-utriv}.  In fact, the study of unit groups is very important in number theory, where one can find literally hundreds of papers relating to groups of units of algebraic number fields.  See for example \cite[II.1: 1.20a]{FrTay-ANT} where the unit group of a Dedekind domain fits into an exact sequence of groups. It is also very important in algebraic geometry (see e.g. \cite{Fo-unitaffine}), where connections occur with Weil divisors, the Picard group, and the Brauer group when looking at the unit group of the coordinate ring of an affine variety.

A property of varieties or algebraic sets should admit a\emph{(n  affine) local} version (i.e. one that can be checked on an open affine cover) and a \emph{geometric} version (i.e. one where you tensor up to the algebraic closure first).  I  develop the notion of a \emph{locally fundamental} set (see Definition~\ref{def:FTA}), and show that it behaves in the ways that robust algebraic geometric properties are expected to behave (see e.g. Proposition~\ref{pr:reduceFTA}).  I also develop and show analogous robustness for \emph{locally} (see Definition~\ref{def:locua}), \emph{generically} (see Definition~\ref{def:genua}), and \emph{geometrically unit-additive} (see Section~\ref{sec:geoua}) rings, as well as \emph{(locally) unit-additive schemes} (see Definition~\ref{def:uascheme}).  I show that these properties behave well with respect to geometric and algebraic intuition (see Theorems~\ref{thm:locua}, \ref{thm:geom0}, and \ref{thm:geomp}, Corollary~\ref{cor:connua}, and Propositions~\ref{pr:schemelocua} and \ref{pr:genuafg}), but not perfectly (see Examples~\ref{ex:genuapowerser} and \ref{ex:notgeom}).  Moreover, the dictionary between fundamentality and unit-additivity extends to their local and geometric versions (see Theorem~\ref{thm:FTAiffua}).

In the final section, we investigate one more geometric tool.  Namely, given a property of varieties or schemes, one wants to be able to check it on irreducible components.  For rings, this amounts to going back and forth between $R$ and $R/\p$ for minimal primes $\p$.  I show that in the case of unit-additivity, this only works one way.  In particular, if $R/\p$ is unit-additive for all minimal primes $\p$, then so is $R$ (see Proposition~\ref{pr:modminprimes}).  But I give an example where the converse fails, even when $R$ is the coordinate ring for an affine variety over an algebraically closed field (see Example~\ref{ex:cantdescendcomponents}).

\section{Algebraic varieties that satisfy the Fundamental Theorem of Algebra}\label{sec:fund}

As there are some discrepancies in the literature over the terms in the case of a non-algebraically closed field, I record the following definitions here for use in the remainder of the paper:

\begin{defn}
Let $k$ be a field and $n \in \N$.  Then an \emph{affine algebraic set} $V \subseteq k^n$ (or an \emph{algebraic subset} of $k^n$) is a set of the form $V(Z) := V_k(Z) := \{p \in k^n \mid f(p) = 0 \text{ for all } f \in Z\}$, where $Z$ is some subset of $k[X_1, \ldots, X_n]$.

The \emph{Zariski topology} on $k^n$ is the topology whose closed sets are the algebraic subsets of $k^n$.  If $V \subseteq k^n$ is an affine algebraic set, then the Zariski topology on $V$ is the subspace topology induced by that of $k^n$.

For any subset $Y$ of $k^n$, $I(Y) := I_k(Y) := \{f \in k[X_1, \ldots, X_n] \mid f(p)=0 \text{ for all } p\in Y\}$.
\end{defn}

The following is presumably well-known.
\begin{lemma}\label{lem:closure}
Let $k$ be a field, let $V \subseteq k^n$ be an affine algebraic set, and let $I = I_k(V) \subseteq k[X_1, \ldots, X_n]$.  Let $W=V_{k^\alg}(Ik^\alg[X_1, \ldots, X_n]) = $ the $n$-tuples with entries in $k^\alg$ at which the generators of $I$ vanish.  Then $W$ is the Zariski closure of $V$ in $(k^\alg)^n$.
\end{lemma}

Next note the following topological exercise.

\begin{lemma}
Let $C,T$ be topological spaces such that $C$ satisfies the  $(T_1)$ separation axiom.  Let $V$ be a subspace of $T$, and let $W$ be the closure of $V$ in $T$.  Let $g: T \ra C$ be a continuous map. \begin{enumerate}
    \item If $V$ is connected, so is $W$.
    \item $g$ is constant on $V$ if and only if it is constant on $W$.
    \item Assume $V$ has only finitely many connected components.  Then $g$ is locally constant on $V$ if and only if it is locally constant on $W$. 
\end{enumerate}
\end{lemma}

The next result follows directly from the previous two lemmas.
\begin{cor}\label{cor:VtoW}
Let $k$ be a field, $V \subseteq k^n$ an affine algebraic set, $I = I(V) \subseteq k[X_1, \ldots, X_n]$, and $W = V(Ik^\alg[X_1, \ldots, X_n])$. 
\begin{enumerate}
    \item If $V$ is (Zariski-)connected, then so is $W$, and
    \item For any $g \in k[X_1, \ldots, X_n]$, $g$ is constant (resp. locally constant) on $V$ if and only if it is constant (resp. locally constant) on $W$.
\end{enumerate}
\end{cor}

\begin{defn}\label{def:FTA}
Let $k$ be a field and let $V \subseteq k^n$ be an affine algebraic set.  Let $W$ be the Zariski closure of $V$ in $(k^\alg)^n$.  Say that $V$ is \emph{fundamental} (\emph{over} $k$) if for any $f \in k[X_1, \ldots, X_n]$ that is not constant on $V$, there is some $p\in W$ with $f(p)=0$.

Say that $V$ is \emph{locally} fundamental if for any $f \in k[X_1, \ldots, X_n]$ that is not \emph{locally} constant on $V$, there is some $p\in W$ with $f(p)=0$.  Note that any fundamental set is locally fundamental.
\end{defn}

\begin{rem}
Some remarks are in order. 
\begin{enumerate}
\item The \emph{classical} version of the Fundamental Theorem of Algebra (the one proved by Gauss and d'Alembert in the 18th century) says that for any polynomial $f$ in one variable with coefficients in $\R$, if $f$ is not constant, then $f$ admits a root in $\C$.  In other words, $\R$ is a fundamental algebraic set (over itself).
\item The \emph{more commonly quoted} version of the Fundamental Theorem of Algebra, that $\C$ is an algebraically closed field, says that $\C$ is a fundamental set.
\item Hilbert's weak Nullstellensatz implies that $k^n$ is fundamental over $k$, for any field $k$ and any $n\in \N$.  Indeed, let $f \in R=k[X_1, \ldots, X_n]$ be nonconstant on $k^n$.  Then $f \notin k$, so for degree reasons we have $1\notin fR$.  Then by the weak Nullstellensatz, $f$ vanishes at some point in $(k^\alg)^n$.  Since $k^n$ is Zariski-dense in $(k^\alg)^n$, it follows that $k^n$ is fundamental over $k$. 
\item Every fundamental set is locally fundamental, but the converse fails.  For instance, consider the algebraic set $S=\{0,1\} = V_k(X^2-X) \subseteq k$, where $k$ is a field of at least three elements. Since $S$ is discrete, \emph{every} function $S \ra k$ is locally constant, whence $S$ is locally fundamental. However, consider the polynomial $g(X) = X-c$, where $c \in k \setminus S$, and the corresponding polynomial function $f$ on $S$.  Then $f$ is nonconstant and nonvanishing on $S$, and $S$ is Zariski-closed in $k^\alg$ (being finite), so $S$ is not fundamental.
\end{enumerate}
\end{rem}

\begin{lemma}\label{lem:FTAconnlocal}
Let $V \subseteq k^n$ be a connected affine algebraic set.  Then $V$ is fundamental if and only if it is locally fundamental.
\end{lemma}

\begin{proof}
The result holds because a function on a connected topological space is constant if and  only if it is locally constant \cite{mse-lcconn}.
\end{proof}

\begin{prop}\label{pr:reduceFTA}
Let $V \subseteq k^n$ be an affine algebraic set.  Let $V = V_1 \coprod \cdots \coprod V_t$ be a decomposition of $V$ into disjoint closed subsets.  Then $V$ is a locally fundamental set if and only if $V_i$ is locally fundamental for all $1\leq i \leq t$.
\end{prop}

\begin{proof}
Let $W$ (resp. $W_i$) be the Zariski closure of $V$ (resp of $V_i$, for each $1\leq i \leq t$) in $(k^\alg)^n$.  Note that $W = W_1 \cup \cdots \cup W_t$.

Suppose all the $V_i$ are locally fundamental.  Let $f \in k[X_1, \ldots, X_n]$ such that $f$ is not locally constant on $V$.  Then there is some $1 \leq i \leq t$ such that $f$ is not locally constant on $V_i$. Since $V_i$ is locally fundamental, there is some $p \in W_i$ (hence $p \in W$) such that $f(p)=0$.

Conversely, suppose $V$ is locally fundamental.  For any $i\neq j$, we have $R = I(\emptyset) = I(V_i \cap V_j) = I(V_i) + I(V_j)$, and $I(V) = I(\bigcup_{j=1}^t V_j) = \bigcap_{j=1}^t I(V_j)$.  Thus, by the Chinese Remainder theorem for rings, we have a ring isomorphism $\phi: k[X_1, \ldots, X_n] / I(V) \rightarrow \prod_{i=1}^t k[X_1, \ldots, X_n] / I(V_i)$ given by $\bar h \mapsto (\bar h, \bar h, \ldots, \bar h)$.  Now, choose $i$ with $1\leq i \leq t$, and let $f \in k[X_1, \ldots, X_n]$ such that $F$ is not locally constant on $V_i$, where $F: k^n \ra k$ is the function induced by the polynomial $f$.  Thus there is some $g\in k[X_1, \ldots, X_n]$ such that $\phi(\bar g) = (1,1,\ldots, 1,\bar f, 1, \ldots, 1)$, with $1$ in every spot except the $i$th, and $\bar f$ in the $i$th spot.  Then $F|_{V_i} = G|_{V_i}$ as functions, where $G: k^n \ra k$ is the function induced by the polynomial $g$.  Let $\tilde F, \tilde G: (k^\alg)^n \ra k^\alg$ be the functions induced by the polynomials $f,g$ respectively on $(k^\alg)^n$.  Since $F$ is not locally constant on $V_i$, neither is $G$, whence $G$ is not locally constant on $V$ either.  Since $V$ is locally fundamental, there is some $p \in W$ such that $\tilde G(p)=0$.  But for any $j \neq i$, we have $G|_{V_j} = 1$, and hence since $\tilde G$ is continuous, $G|_{V_j} = \tilde G|_{V_j}$, and $W_j$ is the Zariski-closure of $V_j$ in $(k^\alg)^n$, we have $\tilde G|_{W_j} =1$.  For the same reason, $\tilde F|_{W_i} = \tilde G|_{W_i}$.  Hence $p \in W_i$, so $0 = \tilde G(p) = \tilde F(p)$.  Thus, $V_i$ is fundamental.
\end{proof}

\begin{prop}\label{pr;embedindependent}
Let $V$ be an affine variety over an infinite field $k$.  Let $S \subseteq k^n$ be an algebraic set such that $S \cong V$ as affine $k$-varieties.  Then $S$ is fundamental (resp. locally fundamental) if and only if for any regular function $g: V \ra \A^1_k$ such that $g \otimes_k k^\alg$ factors through the inclusion map $\G_m(k^\alg) \into \A^1_{k^\alg}$, $g$ is constant (resp. locally constant).  In particular, (local) fundamentality is independent of embedding.
\end{prop}

\begin{proof}
Let $\phi: S \arrow{\cong} V$ be an isomorphism of varieties.  If $\bar{S}$ is the Zariski closure of $S$ in $(k^\alg)^n$, then we have an isomorphism $\bar{\phi}: \bar{S} \arrow{\cong} V \otimes_k k^\alg$.

Suppose $S$ is fundamental.   Let $g: V \ra \A^1_k$ be a regular map such that $g \otimes_k k^\alg$ factors through $\G_m(k^\alg)$.  Let $h \in k[X_1,\ldots, X_n]$ such that for all $s\in S$, $h(s) = g(\phi(s))$.  We have a commutative diagram \[
\xymatrix{
\bar{S} \ar[r]^-{\bar\phi}_-{\cong} & V \otimes_k k^\alg \ar[rr]^-{g \otimes_k k^\alg} \ar[dr]_-f & & \A^1_{k^\alg}.\\
& & \G_m(k^\alg) \ar@{^{(}->}[ur]
}
\]
When $h$ is considered as an element of $k^\alg[X_1, \ldots, X_n]$, then as a function on $\bar S$, it is given by $(g \otimes_k k^\alg) \circ \bar \phi$.  Thus, for any $p \in \bar S$, $h(p) \in \G_m(k^\alg) = k^\alg \setminus \{0\}$ by the above diagram. Thus by fundamentality of $S$, it follows that $h$ is constant on $S$, whence $g$ is constant on $V$.

On the other hand, suppose $S$ is not fundamental.  Then there is some $h\in k[X_1, \ldots, X_n]$ that is nonconstant on $S$, but nonvanishing on $\bar S$.  Set $g := h|_S \circ \phi^{-1}: V \ra \A^1_k$. Then $g$ is nonconstant.  Also, $g \otimes_k k^\alg = h|_{\bar S} \circ \bar \phi^{-1}: V \otimes_k k^\alg \ra \A^1_{k^\alg}$ is nonvanishing. Hence it factors through $\G_m(k^\alg)$.

Now suppose $S$ is locally fundamental.  Write $S = S_1 \coprod \cdots \coprod S_t$, where each $S_i$ is closed and connected.  Then by Proposition~\ref{pr:reduceFTA} and Lemma~\ref{lem:FTAconnlocal}, each $S_i$ is fundamental.  The isomorphism $\phi$ decomposes into a list of isomorphisms $\phi_i: S_i \arrow{\cong} V_i$.  Let $g: V \ra \A^1_k$ be a regular function such that $g \otimes_k k^\alg$ factors through $\G_m(k^\alg)$.  Let $g_i$ be the restriction of $g$ to $V_i$, for each $i$.  Then $g_i \otimes_k k^\alg$ factors through $\G_m(k^\alg)$, so by the first parts of the proof, $g_i$ is constant.  Hence, $g$ is locally constant.

Finally, suppose the condition of the Proposition on local constancy holds for $V$.  Then the corresponding constancy condition holds on each $V_i$, since $V_i$ is connected.  Then by the first parts of the proof, each $S_i$ is fundamental.  By Proposition~\ref{pr:reduceFTA}, $S$ is locally fundamental. 
\end{proof}

This then provides a definition of (locally) fundamental algebraic varieties, or even schemes, over a field $k$.  Just remove the word affine (and replace the word `variety' with `scheme') from the condition above.

\section{Local and generic unit additivity}

\subsection{Background on unit-additivity}
A ring $R$ is said to be \emph{unit-additive} if for any pair of units $u,v$ of $R$ such that $u+v$ is not nilpotent, $u+v$ is a unit (hence the zero ring is vacuously unit-additive).  In \cite{nmeSh-unitadd}, the author and Jay Shapiro introduced and explored the property of unit-additivity, making connections to semigroup rings, irreducible varieties, and elliptic curves.  We also introduced the new invariant of \emph{unit dimension}, which measures how far a domain is from being unit-additive, and the \emph{unit-additive closure} of an integral domain (a universal unit-additive localization).

In this section, we enhance the connection of unit-additivity to algebraic geometry by introducing the related notions of \emph{local} and \emph{generic} unit-additivity, proving along the way some fundamentals of how these properties behave.

Before commencing on the new work, I collect here for ease of reference some results from \cite{nmeSh-unitadd} that I will use in the sequel.

\begin{prop}[{\cite[Proposition 2.18]{nmeSh-unitadd}}]\label{pr:modoutnilpotents}
    Let $R$ be a commutative ring and $I$ an ideal consisting of nilpotent elements.  Then $R$ is \ua\ if and only if $R/I$ is \ua.
\end{prop}

Recall \cite{Cal-UU} that a ring $R$ is \emph{UU} if for any unit $u$ of $R$, $u-1$ is nilpotent.  It is evident that any such ring is \ua\ (see \cite[Proposition 2.12]{nmeSh-unitadd}). On the other hand, we have the following partial converse for products.

\begin{prop}[{\cite[Proposition 2.13]{nmeSh-unitadd}}]\label{pr:uaprod}
    Let $R = S\times T$, where $S,T$ are nonzero commutative rings.  The following are equivalent: \begin{enumerate}
        \item $R$ is \ua.
        \item $S$ and $T$ are UU.
        \item $R$ is UU.
    \end{enumerate}
\end{prop}

It is useful to characterize unit-additivity in various ways.

\begin{prop}[See {\cite[Proposition 2.1]{nmeSh-unitadd}}]\label{pr:characterizeua}
Let $R$ be a nonzero ring.  The following are equivalent: \begin{enumerate}
    \item $R$ is \ua.
    \item For any unit $u$ of $R$, $u+1$ is either a unit or nilpotent.
    \item For any unit $u$ of $R$, $u-1$ is either a unit or nilpotent.
    \item (If $R$ is reduced) The set $U(R) \cup \{0\}$, with structure inherited from $R$, is a field, called the \emph{field of units} of $R$.
\end{enumerate}
\end{prop}

\begin{proof}
The equivalence of (1) and (2) (and also (4), when $R$ is reduced) is part of \cite[Proposition 2.1]{nmeSh-unitadd}.  The equivalence of (2) and (3) follows from the fact that for any $x\in R$, $x$ is a unit (resp. nilpotent) if and only if $-x$ is
\end{proof}

The first connection between fundamentality and unit-additivity was made via the following two theorems from \cite{nmeSh-unitadd}.  I  generalize both of them in Section~\ref{sec:uafund}.

\begin{thm}[{\cite[Theorem 4.1]{nmeSh-unitadd}}]\label{thm:uafgdomain}
Let $R$ be an integral domain finitely generated over a field $k$.  Let $A=k[t,1/t]$. Let $L$ be the integral closure of $k$ in $R$; then $L$ is a finite extension field of $k$.  If $R$ is unit-additive, then the only $k$-algebra maps from $A$ to $R$ send $t$ to some element of $L$; none of these are injective.  On the other hand, if $R$ is \ua, then there is an injective $k$-algebra map $A \ra R$.
\end{thm}

\begin{thm}[{\cite[Theorem 4.2]{nmeSh-unitadd}}]\label{thm:oldfund}
Let $X=\Spec R$ be the scheme associated to an irreducible variety over an algebraically closed field $k$.  Let $\G_m = \A^1_k \setminus \{0\}$.  If $R$ is unit-additive, then all $k$-scheme morphisms $X \ra \G_m$ are constant.  Otherwise, there is a dominant $k$-scheme morphism $X \ra \G_m$, and all such maps have cofinite image.
\end{thm}

\subsection{Local unit-additivity}
As we shall see, unit-additivity is the ring-theoretic avatar of the property of a space being fundamental.  However, in order to complete the interface, we will need the notions of \emph{local} and \emph{geometric} unit-additivity, in part so that we may make a bridge to scheme theory.
There are many ways one might define local unit-additivity in an affine-local type way.  Fortunately, they all coincide, as seen in Theorem~\ref{thm:locua} and Proposition~\ref{pr:schemelocua}.

\begin{thm}\label{thm:locua}
Let $R$ be a commutative ring.  The following are equivalent. \begin{enumerate}
    \item\label{it:finprod} $R_\red$ is isomorphic to a finite product of \ua\ rings.
    \item\label{it:blua} There is a subset $S \subseteq R$ of $R$ that generates the unit ideal such that for all $f\in S$, $R_f$ is \ua.
    \item\label{it:wlua} There is an affine open cover $\cC$ of $\Spec R$ such that for all $U \in \cC$, $\cO_{\Spec R}(U)$ is \ua.
\end{enumerate}
\end{thm}

\begin{proof}
(\ref{it:finprod}) $\implies$ (\ref{it:blua}): Without loss of generality, we may assume we have equality rather than isomorphism.  Let $N$ be the nilradical of $R$. Write $R_\red = R/N = \prod_{i=1}^t R_i$, with each $R_i$ \ua.  For $1\leq i \leq t$, let $e_i$ be the idempotent such that $R_\red \cdot e_i = R_i$.  That is, $e_i = (0,0, \ldots, 0,1,0, \ldots, 0)$, with $1$ in the $i$th position and $0$ in the $j$th position for each $j\neq i$.  Let $f_i$ be a lift of $e_i$ to $R$, for each $1\leq i \leq t$, and set $f := \sum_{i=1}^t f_i$.  Then $f-1 \in N$, so $R = N + fR$. 
Since $N$ is a subset of the Jacobson radical of $R$, it follows from the Nakayama Lemma that $R=fR$.  For each $1\leq i \leq n$, we have $R_{f_i} / N R_{f_i} \cong (R/N)_{e_i} \cong R_i$, which is \ua.  Since $N R_{f_i}$ is contained in the nilradical of $R_{f_i}$, it follows from Proposition~\ref{pr:modoutnilpotents} that $R_{f_i}$ is \ua.  Hence, we may take our set $S$ as in (\ref{it:blua}) to be $\{f_1, \ldots, f_n\}$.

(\ref{it:blua}) $\implies$ (\ref{it:wlua}): Just let $\cC := \{D(f) \mid f \in S\}$.

(\ref{it:wlua}) $\implies$ (\ref{it:finprod}): 
By Proposition~\ref{pr:modoutnilpotents}, since $\cO_{\Spec R}(U)_\red = \cO_{\Spec R_\red}(U)$, we may assume $R$ is reduced.  Since $\Spec R$ is compact, we may assume $\cC$ is finite.  Write $\cC = \{U_1, \ldots, U_t\}$, where each $U_i$ is affine open.  Set $S_i := \cO_{\Spec R}(U_i)$ for $1\leq i \leq t$, so that each $S_i$ is \ua\ and $U_i = \Spec S_i$.  Since each $U_i$ is compact open and every open subset of $\Spec R$ is a union of open sets of the form $D(f)$, $f\in R$, for each $i$ there is a finite set $E_i \subseteq R$ such that $U_i = \bigcup_{f \in E_i} D(f)$.  Then for each $f\in E_i$, we have a pair of ring homomorphisms $R \ra S_i \ra R_f$ corresponding to the restriction maps in the sheaf $\cO_{\Spec R}$, whose composition is therefore the localization map $R \ra R_f$.  Set $E := \bigcup_{1 \leq i \leq t} E_i$.

Note that $E$ generates the unit ideal of $R$. To see this, let $\p \in \Spec R$.  Then since the $U_i$ cover $\Spec R$, there is some $i$ with $\p \in U_i$.  But then for some $f\in E_i$ (hence $f\in E$), we have $\p \in D(f)$. That is, $f\notin \p$, whence $E \nsubseteq \p$, which means that no prime ideal can contain the set $E$.

Impose a binary relation $\sim$ on $E$ as follows:
$f \sim g$ if for any unit $u$ of $R$, then when one considers the localization maps  $R \ra R_f$ and $R \ra R_g$, either $u-1$ maps to units in both $R_f$ and $R_g$, or $u-1$ maps to nilpotent elements in both $R_f$ and $R_g$.

I claim that $\sim$ is an equivalence relation.  It is obviously symmetric and transitive.  To see that it is reflexive, we just need that for all $f\in E$, and any unit $u$ of $R$, $u-1$ either maps to a unit or a nilpotent element in the localization $R_f$.  To see this, let $f \in E$.  Then there is some $1 \leq i \leq t$ such that $f \in E_i$.  Since $R \ra S_i$ is a ring homomorphism and $u$ is a unit in $R$, it is a unit in $S_i$.  Hence since $S_i$ is \ua, $u-1$ is either a unit or nilpotent in $S_i$.  But since $S_i \ra R_f$ is a ring homomorphism, $u-1$ is then either a unit or nilpotent in $R_f$, and since the composition $R \ra S_i \ra R_f$ is the localization map $R \ra R_f$, we are done.

Hence, $\sim$ induces a partition of the set $E$, which is finite since $E$ is a finite set.  Write $E = C_1 \coprod C_2 \coprod \cdots \coprod C_s$, where ``$\coprod$'' means disjoint union and each $C_j$ is an equivalence class under $\sim$.  It follows from the above argument on reflexivity that for any $1\leq i \leq t$ and any $f\in E_i$, if $f \in C_j$ then $E_i \subseteq C_j$.

For each $1\leq j \leq s$, write $W_j := \bigcup \{U_i \mid E_i \subseteq C_j\}$. Since each $E_i$ is contained in some $C_j$, we have $\bigcup_{j=1}^s W_j = \bigcup_{i=1}^t U_i = \Spec R$.  I claim that for any $j \neq k$, we have $W_j \cap W_k = \emptyset$.

To see this, suppose there is some element $\p \in W_j \cap W_k$, where $j \neq k$.  Then there exist $h,i$ and $f,g$ with $f \in E_h \subseteq C_j$, $g \in E_i \subseteq C_k$, and $\p \in D(f) \cap D(g)$.  Since $f \nsim g$, there is some unit $u$ of $R$ such that without loss of generality, $\frac{u-1}1$ is a unit of $R_f$, but $\frac{u-1}1$ is nilpotent in $R_g$.  From the first condition, it follows that there is some $c \in R$ and some $\ell \in \N$ with $f^\ell = (u-1)c$.  Since $f \notin \p$, it follows that $u-1 \notin \p$.  However, from the second condition, there exist $\ell', \ell'' \in \N$ such that $g^{\ell'} (u-1)^{\ell''} = 0 \in \p$.  But $g \notin \p$ and $u-1 \notin \p$, which contradicts the fact that $\p$ is prime.

Thus, $\Spec R$ is the disjoint union of the $W_j$. Hence, for each $1 \leq j \leq s$, we have \begin{align*}
W_j &= \Spec R \setminus \bigcup_{h \neq j} W_h = \Spec R \setminus \bigcup\{U_i \mid E_i \nsubseteq C_j\} \\
&= \bigcap \{V(f) \mid f \in E_i \text{ for some } i \text{ with } E_i \nsubseteq C_j\} \\
&= \bigcap \{V(f) \mid f \in E \setminus C_j\}= V(I_j),
\end{align*}
where $I_j$ is the ideal of $R$ generated by $E \setminus C_j$.

Since $V(I_j+I_k) = V(I_j) \cap V(I_k) = W_j \cap W_k = \emptyset$ for each $j\neq k$, we have that $I_j + I_k = R$ for each such pair.  That is, the ideals $I_j$ are pairwise comaximal.

Also, since $\Spec R = \bigcup_{j=1}^s W_j = \bigcup_{j=1}^s V(I_j) = V(\bigcap_{j=1}^s I_j)$, it follows that $\bigcap_{j=1}^s I_j$ is contained in every prime ideal of $R$, and hence in the nilradical of $R$.  Thus since $R$ is reduced, $\bigcap_{j=1}^s I_j=0$.

It now follows from the Chinese Remainder Theorem for commutative rings \cite[Proposition 1.10
]{AtMac-ICA} that the map \[
\phi: R \ra \prod_{j=1}^s \frac R {I_j}
\]
given by $r \mapsto (r+I_1, r+I_2, \ldots, r+I_s)$ is an isomorphism of rings.

As our final step, I will show that each component ring $R/I_j$ is \ua. First note that the set $B_j := \{f+I_j \mid f\in C_j\} \subset R/I_j$ generates the unit ideal of $R/I_j$.  This is because 
$(B_j) = (C_j) R/I_j = ((C_j) + (E \setminus C_j)) / (E \setminus C_j) = (E) / (E \setminus C_j) = R/I_j$.  Next, note that for any $f \in C_j$ and any $k \neq j$, we have $(R/I_k)_{\bar f} \cong R_f / I_k R_f = 0$, since $f \in I_k$.  Thus, we have \[
\left(\prod_{k=1}^s (R/I_k)\right)_{(\bar f, \bar f, \ldots, \bar f)} \cong \prod_{k=1}^s (R/I_k)_{\bar f} \cong  (R/I_j)_{\bar f}.
\]
Hence, $\phi$ induces an isomorphism between $R_f$ and $(R/I_j)_{\bar f}$.

Now let $u$ be a unit of $(R/I_j)$.  Let $v$ be the preimage via $\phi$ of $(1,1, \ldots, 1, u, 1, \ldots, 1)$, with $u$ in the $j$th spot and $1$ in every other spot.  Then since that tuple is a unit in the product, it follows that $v$ is a unit of $R$.  Let $f \in C_j$.  Then either $v-1$ is a unit or nilpotent in $R_f$.

If $v-1$ is a unit in $R_f$, then since $f \sim g$ for all $g \in C_j$, it follows that $v-1$ is a unit in $R_g$ for all such $g$.  Thus, $u-1$ is a unit in $(R/I_j)_{\bar g}$ for all $\bar g \in B_j$.  But since being a unit is a local property and $B_j$ generates the unit ideal of $R/I_j$, it follows that $u-1$ is a unit of $R/I_j$.

If $v-1$ is nilpotent in $R_f$, then since $f \sim g$ for all $g \in C_j$, it follows that $v-1$ is nilpotent in $R_g$ for all such $g$.  Thus, $u-1$ is nilpotent in $(R/I_j)_{\bar g}$ for all $\bar g \in B_j$.  But since being nilpotent is a local property and $B_j$ generates the unit ideal of $R/I_j$, it follows that $u-1$ is nilpotent in $R/I_j$.

Since $u$ was an arbitrary unit of $R/I_j$, it follows that $R/I_j$ is a \ua\ ring, which then finishes the proof of (\ref{it:finprod}).
\end{proof}

\begin{defn}\label{def:locua}
    If a ring $R$ satisfies one (hence all) of the conditions of Theorem~\ref{thm:locua},  call it \emph{locally \ua}.
\end{defn}

\begin{rem}\label{rem:ref}
The referee pointed out a different way to prove that (\ref{it:blua}) $\implies$ (\ref{it:finprod}), by developing and using elementary facts about idempotent elements.  In particular, a partition somewhat like the sets $C_j$ above are found, and then one shows that the ideals they generate are also generated by orthogonal idempotents. The enterprising reader may try their hand at reproducing it!  However, such an approach does not seem to lead to a shorter proof that (\ref{it:wlua}) $\implies$ (\ref{it:finprod}), which is the implication I need in order for the scheme-theoretic definition of local unit-additivity to work the way one wishes (see Definition~\ref{def:uascheme} and the ensuing discussion).

The referee also points out that of these three definitions of unit-additivity that are equivalent in the standard axiomatic framework, the `best' one from the point of view of constructive mathematics (i.e. avoiding the Law of the Excluded Middle) seems to be (\ref{it:blua}), as the next two corollaries follow quite quickly from it.  A constructive version of this paper would be interesting, but is beyond my expertise.
\end{rem}

Note that any \ua\ ring is locally \ua, as it is the product of a \emph{single} \ua\ ring.

\begin{cor}\label{cor:connua}
Let $R$ be a ring with connected spectrum.  Then $R$ is \ua\ $\iff$ it is locally \ua.
\end{cor}

\begin{proof}
Any \ua\ ring is locally \ua\ by Theorem~\ref{thm:locua}(\ref{it:blua}), with $S = \{1\}$.

Conversely, suppose $R$ is locally \ua.  By Theorem~\ref{thm:locua}(\ref{it:finprod}), $R_\red$ is a finite product of \ua\ rings.  Write $R_\red = R_1 \times \cdots \times R_t$, with each $R_i$ \ua.  But if each $R_i$ is a nonzero ring, then since $\Spec R_\red$ is homeomorphic to $\Spec R$ and the latter is connected, it follows that $t=1$.  Thus $R_\red$ is \ua, whence $R$ is \ua\ by Proposition~\ref{pr:modoutnilpotents}. \end{proof}

\begin{cor}\label{cor:locuaprod}
Let $S, T$ be rings and $R=S \times T$.  Then $R$ is locally \ua\ if and only if $S$ and $T$ are locally \ua.
\end{cor}

\begin{proof}
Note that $R_\red = S_\red \times T_\red$.  Thus, the ``if'' direction follows directly from Theorem~\ref{thm:locua}(\ref{it:finprod}).  Since for any $f=(g,h) \in R$, we have $R_f \cong S_g \times T_h$, the converse follows from Theorem~\ref{thm:locua}(\ref{it:blua}) and 
Proposition~\ref{pr:uaprod}.
\end{proof}

\begin{rem}
Local unit-additivity does not imply unit-additivity.  Let $S$, $T$ be \ua\ rings and such that $S$ is not UU (i.e., $S_\red$ has at least two distinct units), and let $R := S \times T$.  Then $R$ is locally \ua\ by Corollary~\ref{cor:locuaprod}, but $R$ is not \ua\ by Proposition~\ref{pr:uaprod}.
\end{rem}

Next, we generalize to schemes.
\begin{defn}\label{def:uascheme}
An affine scheme is \emph{\ua} if it is the spectrum of a \ua\ ring.
A scheme $X$ is \emph{locally \ua} if it has a cover by affine open \ua\ subschemes.
\end{defn}

\begin{rem} 
Note that the above terminology is compatible with that of Definition~\ref{def:locua}. That is, a ring $R$ is locally \ua\ if and only if $\Spec R$ is a locally \ua\ scheme.  This follows from using  Theorem~\ref{thm:locua}(\ref{it:wlua}) as our working definition of local unit-additivity of rings.
\end{rem}

\begin{prop}\label{pr:schemelocua}
    Let $X$ be a quasicompact scheme with only finitely many connected components (e.g. any Noetherian or affine connected scheme), written $X = X_1 \coprod \cdots \coprod X_n$. The following are equivalent. \begin{enumerate}[(a)]
        \item $X$ is locally \ua.
        \item Each $X_j$ is locally \ua.
    \end{enumerate}
    If $X$ is itself affine, then another equivalent condition is: \begin{enumerate}[(c)]
        \item Each $X_j$ is \ua.
    \end{enumerate}
\end{prop}

\begin{proof}
(b) $\implies$ (a): Take the union of selected open affine \ua\ covers of $X_i$ for each $1\leq i \leq n$.

(a) $\implies$ (b): By compactness, we can replace the given cover by a finite subcover, so $X = U_1 \cup \cdots \cup U_m$, with each $U_i$ open, affine, and \ua.  For each pair $i,j$, we have that $U_i \cap X_j$ is an open affine subscheme of both $U_i$ and $X_j$, and $\cO_X(U_i) \cong \prod_{j=1}^n \cO_X(U_i \cap X_j)$.  Then by Proposition~\ref{pr:uaprod}, each $\cO_X(U_i \cap X_j)$ is \ua. Thus for each $j$, we have that $\{U_i \cap X_j \mid 1\leq i \leq m\}$ is an open affine \ua\ cover of $X_j$.

Now we specialize to the case $X=\Spec R$ is affine.

(c) $\implies$ (a): This is clear since each $X_j$ is affine and open in $X$.

(b) $\implies$ (c): This follows from Corollary~\ref{cor:connua}.
\end{proof}

\begin{rem}\label{rem:refconn}
Even when $X$ is affine, the $U_i$ in the proof above need not be connected.  Indeed, any finite product of UU rings (e.g., $\F_2[x,y]$, $\F_2$, $\Z / 8\Z$, etc.) is \ua\ by Proposition~\ref{pr:uaprod}, so for instance if $R=\F_2 \times \F_2[t]$, we can let $X=\Spec R = U$, and $U$ is an affine open unit-additive subset of $X$, but $X= U=\Spec \F_2 \coprod \Spec \F_2[t]$ has two connected components. However, this subtlety does not arise for algebras over a field of at least three elements.
\end{rem}

\subsection{Generic unit-additivity}

In geometry and topology, it is interesting to know when a property holds \emph{generically} -- i.e., on ``nearly all'' of the space in question.  In commutative algebra, this is typically measured by localizing at a single element to get the equivalent of a dense open set.  Hence, consider the following.
\begin{defn}\label{def:genua}
An integral domain $R$ is \emph{generically \ua} if there is some nonzero element $a \in R$ such that $R_a$ is \ua.
\end{defn}

\begin{example}\label{ex:genuapowerser}
Generic unit-additivity can be weaker than local unit-additivity.  For example, let $R = k[\![X]\!]$. Then $R$ is not \ua\ (see \cite[Example 2.6]{nmeSh-unitadd}), and hence is not locally \ua\ (see Corollary~\ref{cor:connua}), but it is generically \ua\ since $R[1/X] = k(\!(X)\!)$ is a field.
\end{example}

However, the above phenomenon cannot occur with domains that are finitely generated algebras over a field.

\begin{prop}\label{pr:genuafg}
If $R$ is a generically \ua\ domain that is finitely generated over a field $k$, then it is \ua.
\end{prop}

\begin{proof}
To prove the contrapositive, suppose $R$ is not \ua.  Let $A = k[t, t^{-1}]$.  By Theorem~\ref{thm:uafgdomain}, there is an injective $k$-algebra map $\phi: A \ra R$.  Let $0 \neq x \in R$.  Let $\ell: R \ra R_x$ be the localization map.  Then $R_x \cong R[y] / (xy-1)$ is a domain that is finitely generated as a $k$-algebra, and $\ell \circ \phi$ is an injective $k$-algebra map.  Thus, by Theorem~\ref{thm:uafgdomain}, $R_x$ is not \ua.  Since $x$ was arbitrary, $R$ is not generically \ua.
\end{proof}

In fact, generic unit-additivity has an interesting connection with G-domains, as in \cite{Kap-CR}.  Recall that a \emph{G-domain} is an integral domain $R$ that admits a nonzero element $f$ such that $R[1/f]$ is a field (necessarily the fraction field of $R$).  

\begin{prop}[Shapiro]\label{pr:Gdom}
Let $R$ be a reduced ring and $0\neq f \in R$.  Then $R_f$ is \ua\ if and only if there is a subring $A$ of $R$ containing $f$ such that the following three conditions hold: \begin{enumerate}
    \item The ideal $I := \bigcup_{n \in \N} (0 : f^n)$ of $R$ is a prime ideal of $A$.
    \item Every element $d$ of $R/I$ that divides a power of $f$ (in $R/I$) is an element of $A/I$.
    \item $A/I$ is a G-domain and $A_f \cong (A/I)_{\bar f}$ is a field.
\end{enumerate}
If so, then $A_f$ is the field of units of $R_f$.
\end{prop}

\begin{proof}
Suppose $R_f$ is \ua.  Set $I :=\bigcup_{n \in \N} (0 : f^n)$.  Then $I$ is the kernel of the localization map $R \ra R_f$, so one can see $R/I$ as a subring of $R_f$. Note that $k := U(R_f) \cup \{0\}$ is a field by Proposition~\ref{pr:characterizeua}.  Let $A' := k \cap R/I$, and let $A$ be the preimage of $A'$ via the natural projection $R \onto R/I$.  Then $I \subseteq A$, and $A/I \cong A' \hookrightarrow k$. Since $k$ is an integral domain, so is it subring $A'$, so that $I$ is a prime ideal of $A$.  Now let $d,e \in R$ and $n\in \N$ with $de-f^n \in I$.  Then $d/1$ is a unit of $R_f$, whence $\bar d \in k \cap R/I = A/I$. It remains to show that the induced map $A'_{\bar f} \hookrightarrow k$ is surjective.  To see this, let $c/f^n \in k$, $c \in R$.  Then by (2), $\bar c \in A'$, so $c/f^n = \bar c / \bar f^n\in A'_{\bar f}$.  But also $A_f \cong A'_{\bar f}$ since $I$ is the kernel of $A \ra A_f$, which means that $A_f$ is the field of units of $R_f$.

Conversely, suppose the given three conditions hold. Let $v$ be a unit of $R_f$.  Write $v = u/f^n$ with $u\in R$.  Then there exist $d\in R$ and $m \in \N$ with $\bar d \bar u = \bar f^{n+m}$ in $R/I$.  Then by (2), $\bar u$ is an element of $A/I$.  Thus, $v=u/f^n \in A_f =:k$, which is a field by (3).  Thus, all units of $R_f$ are in a subfield of $R_f$, whence $R_f$ is \ua\ by Proposition~\ref{pr:characterizeua}.
\end{proof}

\begin{cor}[Shapiro]\label{cor:hasGdom}
Let $R$ be an integral domain that is generically \ua\ but not \ua.  Then $R$ contains a G-domain that is not a field.
\end{cor}

\begin{proof}
Let $0\neq f \in R$ such that $R_f$ is \ua.  Let $A$ and $I$ be as in Proposition~\ref{pr:Gdom}. Then $I=0$ since $f$ is not a zero-divisor of $R$. 
Since $R$ is not \ua, there is some unit $u \in R$ such that $u+1$ is a nonzero nonunit.  Then $u/1$ is a unit of $R_f$, which means there is some $r\in R$ and $n\in \N$ with $ur=f^n$. 
 But then $u\in A$ by Proposition~\ref{pr:Gdom}. 
Since $u+1$ is a nonzero nonunit of $R$, it must be a nonzero nonunit of the subring $A$.  Thus, $A$ is not a field, but $A_f$ is, so $A$ is a G-domain.
\end{proof}

\section{Unit-additive rings correspond to fundamental varieties}\label{sec:uafund}

In \cite[Section 4]{nmeSh-unitadd}, the author and Jay Shapiro proved a theorem characterizing \ua\ rings among \emph{integral} finitely generated $k$-algebras.
In this section, I generalize these results to the reduced case  (see Theorem~\ref{thm:new41}), hence encompassing all algebraic varieties, irreducible or not.  I then use this to show that an affine variety is (locally) fundamental if and only if its coordinate ring is (locally) unit-additive (see Theorem~\ref{thm:FTAiffua}).

First, consider the following:

\begin{lemma}\label{lem:fieldconn}
Let $R$ be a reduced ring with connected spectrum and only finitely many minimal primes.  Let $k$ be a field that is a subring of $R$, and let $L$ be the integral closure of $k$ in $R$.  Then $L$ is a field that is contained in all the residue fields of $R$.  If $R$ is finitely generated as a $k$-algebra, then $L$ is the unique maximal subfield of $R$.
\end{lemma}

\begin{proof}
Since $L$ is integral over $k$, it has Krull dimension zero.  It is also reduced, since it is a subring of $R$.

Now let $\p_1, \ldots, \p_n$ be the minimal prime ideals of $R$. Let $\q_i = \p_i \cap L$ for each $1 \leq i \leq n$. By reordering, we may choose $t$ with $1 \leq t \leq n$ such that $\q_1, \ldots, \q_t$ are distinct and for each $i>t$, there is some $j \leq t$ with $\q_i = \q_j$.  Then $\bigcap_{i=1}^t \q_i = \bigcap_{i=1}^n \q_i \subseteq \bigcap_{i=1}^n \p_i = 0$, and the $\q_i$ (for $i \leq t$) are pairwise comaximal since $\dim L=0$.  Hence \cite[Proposition 1.10]{AtMac-ICA} the mapping $L \rightarrow \prod_{i=1}^t L/\q_i$ is an isomorphism.  But if  $t>1$, it follows that $L$ contains a nontrivial idempotent, whence $R$ does as well, contradicting the fact that $\Spec R$ is connected.  Hence $L \onto L/\q_1$ is an isomorphism, whence $\q_1=0$ and $L$ is a field.

Let $\kappa$ be a residue field of $R$.  Then we have a sequence of ring homomorphisms $L \into R \onto \kappa$, whose composition, since $L$, $\kappa$ are fields, amounts to a field extension.

Now assume $R$ is finitely generated as a $k$-algebra. Let $F$ be a subfield of $R$.  Let $\m \in \Max R$.  Then the composition $k \ra R \ra R/\m$ is a finite algebraic extension by 
Zariski's lemma. Let $0 \neq c \in F$. Then the image of $c$ in the composition $F \ra R \ra R/\m$ is nonzero since $\m \cap F=0$, so there is a nonzero (hence one that can be replaced by a monic) polynomial $g \in k[X]$ such that $g(c)=0$.  Thus, $c\in L$, so $F \subseteq L$.
\end{proof}

\begin{example}[Shapiro]\label{ex:nouniquemaximal}
    Note that connected spectrum is a necessary condition in the previous result.  In particular, let $k=\R$ and $A = \C \times \C$, with $k$ identified as a subring of $A$ by the diagonal embedding.  Consider the elements $a=(i,i)$ and $b=(-i,i)$ of $A$.  Then $k[a] \cong \C \cong k[b]$ as $\R$-algebras. Since any field contained in $A$ must be a finite algebraic extension of $k$ (since $\dim_k A =4$), whereas $\C$ is algebraically closed, it follows that $k[a]$ and $k[b]$ are maximal subfields of $A$.  Since $b \notin k[a]$, it follows that there is no \emph{unique} maximal subfield of $A$, nor is there a unique maximal element among the fields in $A$ that contain $k$.  Moreover, $A$ is integral over $k$, and hence the integral closure of $k$ in $A$ is not a field.
\end{example}

Instead, for the disconnected spectrum case, we have the following.

\begin{prop}\label{pr:disconnbreakup}
Let $R$ be a reduced ring with only finitely many minimal primes.  Let $k$ be a field that is a subring of $R$.    Then there is a ring isomorphism $\phi: R \ra R_1 \times \cdots \times R_n$, with each $R_i$ a ring with connected spectrum that contains $k$.  Let $L$ be the integral closure of $k$ in $R$.  Then we have $\phi(L) = \prod_{i=1}^n L_i$, where for each $1\leq i \leq n$, $L_i$ is a field.  If moreover $R$ is a finitely generated $k$-algebra, then for each $i$, $L_i$ is the unique maximal subfield of $R_i$.
\end{prop}

\begin{proof}
Since each minimal prime corresponds to an irreducible component of $\Spec R$, and every irreducible component of $\Spec R$ is connected, it follows that $R$ has only finitely many connected components.  This then corresponds to a list of nontrivial idempotent elements $e_1, \ldots, e_n$ of $R$ such that $e_i e_j = 0$ for all $i<j$ and $\sum_{i=1}^n e_i = 1$.  This then gives an isomorphism $\phi: R \ra \prod_{i=1}^n R_i$ as claimed.  Note that for each $i$, $e_i$ is the multiplicative identity of $R_i$. 

Since for each $i$, $e_i$ is a root of the monic polynomial $X^2 - X \in k[X]$, $e_i$ is integral over $k$ -- i.e., $e_i \in L$.  Then $L_i = \phi(L e_i)$ is the integral closure of $\phi(k)$ in $R_i$, so by Lemma~\ref{lem:fieldconn}, $L_i$ is a field.  If $R$ is finitely generated over $k$, then each $R_i$ is finitely generated over $\phi(k)$, so again by Lemma~\ref{lem:fieldconn}, $L_i$ is the unique maximal subfield of $R_i$.
\end{proof}

The following is a generalization of Theorem~\ref{thm:uafgdomain}.
\begin{thm}\label{thm:new41}
Let $R$ be a finitely generated reduced $k$-algebra with connected prime spectrum, where $k$ is a field.  Let $L$ be the integral closure of $k$ in $R$.  Let $A=k[t, t^{-1}]$, where $t$ is an indeterminate over $k$.  Then exactly one of the following is true:
\begin{enumerate}
    \item $R$ is \ua, and the only $k$-algebra homomorphisms from $A$ to $R$ are the ones that send $t \mapsto \alpha$ for some $\alpha \in L^\times$ (i.e., $U(R) = L^\times$). None of these are injective.
    \item $R$ is not \ua, and there is an injective $k$-algebra map $\phi: A \ra R$. 
\end{enumerate}
\end{thm}

\begin{proof}
If $R$ is \ua, the proof of (1) is identical to the proof of Theorem~\ref{thm:uafgdomain}, using the fact that by Lemma~\ref{lem:fieldconn}, $L$ is the unique maximal subfield of $R$ containing $k$.

If $R$ is not \ua, then again the proof of (2) is identical to the proof of Theorem~\ref{thm:uafgdomain}.
\end{proof}

Next, I generalize Theorem~\ref{thm:oldfund}.
\begin{thm}\label{thm:new42}
Let $X$ be a (not necessarily connected or irreducible) affine variety over an algebraically closed field $k$. Let $\G_m = \A^1_k \setminus \{\mathbf 0\}$.  Then exactly one of the following is true:
\begin{enumerate}
    \item $X$ is locally \ua, and all $k$-scheme morphisms $X \ra \G_m$ are locally constant.
    \item $X$ is \emph{not} locally \ua, and there is a dominant $k$-scheme morphism $X \ra \G_m$.  Any such map has cofinite image.
\end{enumerate}
\end{thm}

\begin{proof}
Write $X = \Spec R$, $R = R_1 \times \cdots \times R_n$ with each $R_i$ a ring with connected spectrum, and $X_i = \Spec R_i$. Let $\gamma_i: X_i \into X$ be the inclusion maps.  

Suppose $X$ is locally \ua. By Proposition~\ref{pr:schemelocua}, each $R_i$ is \ua.  Let $g: X \ra \G_m$ be a $k$-scheme morphism and $g_i = g \circ \gamma_i$.  Then by Theorem~\ref{thm:new41}, $g_i(t) = \alpha_i$ for some $\alpha_i \in k^\times$.  Thus, $g_i$ is the constant map that sends every $x \in X_i$ to $\alpha_i$.  Hence, $g$ is locally constant.

Now suppose $X$ is not locally \ua.  Then by Proposition~\ref{pr:schemelocua}, there is some $1\leq j \leq n$ with $X_j$ not \ua.  By Theorem~\ref{thm:new41}, there is an injective $k$-algebra map $\phi: k[t, t^{-1}] \ra R_j$.  Let $\psi: k[t, t^{-1}] \ra R$ be the $k$-algebra morphism such that $\psi(t) = (1,\ldots, 1, \phi(t),1, \ldots, 1)$.  Then $\psi$ is injective, so the corresponding $k$-scheme morphism $X \ra \G_m$ is dominant.

Now, let $\mu: X \ra \G_m$ be any dominant $k$-scheme morphism.  Since $\G_m$ is irreducible, it follows that there is some irreducible component $Y$ of $X$ such that $\mu \circ i: Y \ra \G_m$ is dominant, where $i: Y \into X$ is the inclusion.  Then by Theorem~\ref{thm:oldfund}(2), the image of $\mu \circ i$ is cofinite.  Hence the image of $\mu$ is also cofinite. 
\end{proof}

\begin{thm}\label{thm:FTAiffua}
Let $V \subseteq k^n$ be an affine algebraic set, where $k$ is an infinite field.  Let $I=I(V)$ and $R = k[X_1, \ldots, X_n]/I$.  Then $V$ is fundamental (resp. locally fundamental) if and only if $R$ is \ua\ (resp. locally \ua).
\end{thm}

\begin{proof}
By Lemma~\ref{lem:FTAconnlocal}, Proposition~\ref{pr:reduceFTA}, and Corollaries~\ref{cor:connua} and \ref{cor:locuaprod}, we may assume $V$ is connected and dispense with the adverb ``locally''.

Suppose $V$ is a fundamental set. Let $r \in R$ be a unit.  Let $F \in k[X_1, \ldots, X_n]$ with $r = F+I$.  Set $S := R \otimes_k k^\alg$, and let $W$ be the Zariski closure of $V$ in $(k^\alg)^n$.  Then since $R=rR$, we have $S/FS \cong (R/rR) \otimes_k k^\alg=0$, so $s := \bar F \in S$ generates the unit ideal of $S$.  Hence, $F$ is nonvanishing on $W$, so by assumption, $F$ is constant on $V$. Say $\lambda \in k^\times$ such that $F(p) =\lambda$ for all $p \in V$.  Then $F-\lambda \in I(V) =I$, so that in $R$, we have $r=\bar F = \lambda \in k$.  Hence $r+1\in k$ is either a unit or zero, so $R$ is \ua.

Conversely suppose $R$ is \ua.  Let $F \in k[X_1, \ldots, X_n]$ be a polynomial that does not vanish anywhere on $W$.  Then by the weak Nullstellensatz, $\bar F$ is a unit in $S$.  Thus, \[
0 = S/\bar FS \cong (R/\bar FR) \otimes_k k^\alg.
\]
Then since $k^\alg$ is faithfully flat over $k$, we have $R / \bar FR = 0$, so that $\bar F$ is a unit of $R$.  But $R$ is \ua.  Moreover, $k$ is the field of units of $R$. To see this, let $L$ be the integral closure of $k$ in $R$, which by Theorem~\ref{thm:new41} is the field of units of $R$.  Let $p\in V$. Then $k \cong R/I(p)$ is a residue field of $R$, so by Lemma~\ref{lem:fieldconn} we have $k \subseteq L \subseteq k$, whence $k=L$.  Therefore, $\bar F \in k^\times$ in $R$, so that $F$ is constant on $V$.  Thus, $V$ is a fundamental set.
\end{proof}

\begin{question}
    What about non-affine varieties?  If $k$ is an algebraically closed field and $X$ is a $k$-variety such that all nonvanishing $k$-scheme morphisms $X \ra \A^1_k$ are locally constant, does it follow that $X$ is locally \ua?
    
    It is classical that for any irreducible projective variety $X$ over an algebraically closed field $k$, the only $k$-scheme morphisms from $X$ to $\A^1_k$ (hence also all $k$-scheme morphisms to $\G_m(k)$) are constant.  So are all such varieties locally \ua?
\end{question}

\section{Geometric unit-additivity and base change}\label{sec:geoua}
A basic tenet of modern algebraic geometry is that with any important property, one should learn how it behaves under base change by field extensions.  In this section, we will do so (at the affine level) with unit-additivity, which then leads to the notion of \emph{geometric unit-additivity over a field}.
\begin{defn}
Let $R$ be a ring containing a field $K$, and let $r \in R$.  Say $r$ is \emph{\pins\ over $K$} if either $r\in K$, or else $\chr R = p>0$ and $r^{p^n} \in K$ for some $n \in \N$.
\end{defn}

\begin{rem}\label{rem:sep}
If $\chr R = p>0$, then even if $K$ is perfect (or even algebraically closed), one can have elements of $R \setminus K$ \pins\ over $K$, provided $R$ is not reduced.  In fact, when $K$ is perfect, then \emph{$r\in R$ is \pins\ over $K$ if and only if $r$ is nilpotent mod $K$.}

To see this, first note that if $c \in K$ such that $r-c$ is nilpotent, then there is some $n$ with $(r-c)^{p^n} = 0$, so $r^{p^n} = c^{p^n} \in K$.  Conversely, suppose $r$ is \pins\ over $K$.  Then there is some $n \in \N$ and $c \in K$ with $r^{p^n} = c$.  Let $b \in K$ be a $p^n$th root of $c$, which exists in $K$ because $K$ is perfect.  Then $r^{p^n} = b^{p^n}$, so $(r-b)^{p^n}=0$, and $r-b$ is nilpotent. Hence $r$ is nilpotent mod $K$.
\end{rem}

\begin{prop}\label{pr:Sawin}
Let $K$ be an algebraically closed field, and let $R$ be a \ua\ ring containing $K$ such that all units of $R$ are in $K$ mod nilpotents.  Let $L$ be a reduced ring that contains $K$.  Then all the units of $S=R \otimes_K L$ are in $L$ mod nilpotents.  Hence, if $L$ is a field, then $S$ is \ua\ with all units in $L$ mod nilpotents.
\end{prop}

\begin{proof}
First suppose $R$ is reduced, so that all units of $R$ are in $K^\times$. Let $u$ be a unit in $S$.  Then there is a sequence $r_0, \ldots, r_n \in R$ linearly independent over $K$, which can be chosen so that $r_0=1$, and elements $b_0, \ldots, b_n \in L$ such that $u = \sum_{i=0}^n r_i \otimes b_i$.  Suppose there is some $j>0$ such that $b_j \neq 0$.

Since $u$ is a unit, there is some $u' \in S$ with $uu' =1$.  We have $u' = \sum_{i=0}^m c_i \otimes d_i$ with $c_i \in R$ and $d_i \in L$.  Set $A := K[b_0, \ldots, b_n, d_0, \ldots, d_m]$.  Note that $A$ is a subring of $L$ that is finitely generated as a $K$-algebra.  Since $b_j \neq 0$ and $A$ is a Hilbert ring, there is some maximal ideal $\m$ of $A$ such that $b_j \notin \m$.  By the Nullstellensatz (since $A$ is reduced), $A/\m \cong K$.  Hence, we have a surjective $K$-algebra map $f: A \onto K$ with $f(b_j) \neq 0$.  Let $g = 1_R \otimes f: R \otimes_K A \rightarrow R$.  Write $B = R \otimes_K A$.  Note that all the specifically named simple tensors above are elements of $B$.  Since $u \in B$ is a unit, $g(u)$ is a unit of $R$.  But then $g(u) \in K^\times$.  We have $g(u) = \sum_{i=0}^n f(b_i) r_i = f(b_0) + \sum_{i=1}^n f(b_i)r_i$.  Then we get a $K$-linear relationship in $R$ among the $r_i$ via $0 = (f(b_0) - g(u)) \cdot 1 + \sum_{i=1}^n f(b_i)r_i$. Hence all the $f(b_i) = 0$ for $i>0$.  But this contradicts $f(b_j) \neq 0$.

Hence, $b_j=0$ for all $j>0$, so $u=b_0 \in L$.

Now let us generalize to the case where $R$ may not be reduced.  Let $N$ be the nilradical of $R$.  Let $u$ be a unit of $S$.  Then $\bar u \in S/NS$ is a unit, and $S/NS \cong (R/N) \otimes_K L$.  Hence by the reduced case above, $\bar u \in L$, so there is some $\lambda \in L$ such that $u-\lambda \in NS$.  Since $NS \subseteq $ the nilradical of $S$, it follows that all the units of $S$ are in $L$ mod nilpotents.

The last statement holds because if $L$ is a field, all its nonzero elements are units of $S$.
\end{proof}

Here is what must be a well-known lemma.

\begin{lemma}\label{lem:tensor}
    Let $K$ be a field, let $A,B$ be $K$-algebras, and let $a\in A$, $b \in B$ such that $a \otimes 1 = 1 \otimes b$ as elements of $A \otimes_K B$.  Then $a=b \in K$.
\end{lemma}

\begin{proof}
Let $\cA$ be a basis of $A$ as a $K$-vector space such that $1\in \cA$.  Let $\cB$ be a basis of $B$ as a $K$-vector space such that $1\in \cB$.  There exist distinct $a_1, \ldots, a_n \in \cA \setminus \{1\}$ and unique $c_0, c_1,\ldots, c_n \in K$ with $a = c_0 + \sum_{i=1}^n c_i a_i$.  Similarly, there exist distinct $b_1, \ldots, b_m \in \cB \setminus \{1\}$ and unique $d_0, d_1, \ldots, d_m \in K$ with $b = d_0 + \sum_{i=1}^m d_i b_i$. Thus, \[
0 = a \otimes 1 - 1 \otimes b = (c_0 - d_0) \cdot (1 \otimes 1) + \sum_{i=1}^n c_i (a_i \otimes 1) - \sum_{i=1}^m d_i (1 \otimes b_i).
\]
Since $\cA \otimes \cB = \{\alpha \otimes \beta \mid \alpha \in \cA, \beta \in \cB\}$ is a basis for $A \otimes_K B$ as a $K$-vector space, it follows that $c_i=0$ for all $i\geq 1$, $d_i=0$ for all $i\geq 1$, and $c_0 = d_0$.  Thus, $a = c_0 = d_0 = b \in K$.
\end{proof}

\begin{lemma}\label{lem:fielddescentchar0}
Let $R$ be a $K$-algebra, where $K$ is a field of characteristic zero, and $L/K$ is a field extension such that $S:=R \otimes_K L$ is \ua\ with all units in $L$ mod nilpotents.  Then all units of $R$ are in $K$ mod nilpotents.
\end{lemma}

\begin{proof}
First assume $R$ is reduced.  Let $u$ be a unit of $R$.  Then $u \otimes 1$ is a unit of $S$.  Since $\chr K=0$, we have that $L/K$ is separable, whence $S$ is reduced.  Thus, there is some $c\in L^\times$ with $u \otimes 1 = 1 \otimes c$.  By Lemma~\ref{lem:tensor}, it follows that $u=c \in K$.

Now pass to the case where $R$ may not be reduced.  Let $N$ be the nilradical of $R$.  Then by Proposition~\ref{pr:modoutnilpotents}, $S/NS$ is \ua.  Moreover, if $u \in S$ such that $\bar u \in S/NS$ is a unit, then $u$ is a unit of $S$, so there is some nilpotent $n\in S$ with $u-n \in L$.  Since $S/NS \cong (R/N) \otimes_K L$ is reduced, it follows that $n \in NS$, whence $\bar u \in L$. Then by the first paragraph of the proof, $R/N$ is \ua\ with field of units $K$.  Hence all units of $R$ are in $K$ mod nilpotents.

To see that $R$ is \ua, let $u$ be a unit of $R$.  Then there is some nilpotent element $n$ with $u-n \in K$.  Thus also $u+1-n = c$ for some $c\in K$. If $c=0$, then $u+1=n$ is nilpotent.  Otherwise $u+1=n+c$ is a unit, being the sum of a nilpotent $n$ and a unit $c$.
\end{proof}

\begin{lemma}\label{lem:fielddescentcharp}
    Let $R$ be a $K$-algebra, $K$ a field of characteristic $p>0$, and $L/K$ a field extension such that $S :=R \otimes_K L$ is 
    \ua\ with 
    all units \pins\ over $L$.  Then $R$ is 
    \ua\ with 
    all units \pins\ over $K$.
\end{lemma}

\begin{proof}
Let $u$ be a unit of $R$.  Then $u \otimes 1$ is a unit of $S$, so by assumption there is some $e \in \N$ and $c \in L^\times$ with $(u \otimes 1)^{p^e} = 1 \otimes c$.  That is, $u^{p^e} \otimes 1 = 1 \otimes c$. Then by Lemma~\ref{lem:tensor}, $u^{p^e} = c \in K$.  Hence $u$ is \pins\ over $K$.
Now, $(u+1)^{p^e} = u^{p^e}+1 = c+1 \in K$, which is either a unit or zero, so $u+1$ is either a unit or nilpotent. Since $u$ was an arbitrary unit of $R$, it follows that $R$ is \ua.
\end{proof}

\begin{thm}\label{thm:geom0}
Let $R$ be a ring and $K \subseteq R$ a field of characteristic zero.  The following are equivalent. \begin{enumerate}
    \item For all field extensions $L/K$, $R \otimes_K L$ is \ua\ with all units in $L$ mod nilpotents.
    \item For all algebraic field extensions $L/K$, $R \otimes_KL$ is \ua\ with all units in $L$ mod nilpotents.
    \item For all finite field extensions $L/K$, $R \otimes_K L$ is \ua\ with all units in $L$ mod nilpotents.
    \item $R \otimes_K K^\alg$ is \ua\ with all units in $K^\alg$ mod nilpotents.
\end{enumerate}
\end{thm}

\begin{proof}
It is obvious that (1) $\implies$ (2) and (2) $\implies$ (3).

Suppose (3), and let $u$ be a unit of $R \otimes_K K^\alg$.  Let $v=1/u$.  Then there is a finite field extension $L/K$ with $u,v \in R \otimes_KL$.  Thus $u$ is a unit of $R \otimes_K L$, so $u-n \in L^\times$ for some nilpotent element $n$ of $R \otimes_K L$, so $u-n \in (K^\alg)^\times$. Hence (4) holds.

Finally suppose (4) holds.  Let $L/K$ be a field extension.  We have $R \otimes_K L^\alg \cong (R \otimes_K K^\alg) \otimes_{K^\alg} L^\alg$, which by assumption and Proposition~\ref{pr:Sawin} is \ua\ with all units in $L^\alg$ mod nilpotents.  But $R \otimes_K L^\alg \cong (R \otimes_K L) \otimes_L L^\alg$, so by Lemma~\ref{lem:fielddescentchar0}, (1) follows.
\end{proof}

\begin{thm}\label{thm:geomp}
Let $R$ be a ring and $K \subseteq R$ a field of characteristic $p>0$.  The following are equivalent:
\begin{enumerate}
    \item For all field extensions $L/K$, $R \otimes_K L$ is \ua\ with all units \pins\ over $L$.
    \item For all algebraic field extensions $L/K$, $R \otimes_KL$ is \ua\ with all units \pins\ over $L$.
    \item For all finite field extensions $L/K$, $R \otimes_KL$ is \ua\ with all units \pins\ over $L$.
    \item $R \otimes_K K^\alg$ is \ua\ with all units in $K^\alg$ mod nilpotents.
\end{enumerate}
\end{thm}

\begin{proof}
As in the characteristic zero case, it is clear that (1) $\implies$ (2) and (2) $\implies$ (3).

Suppose (3), and let $u$ be a unit of $R \otimes_K K^\alg$.  Let $v=1/u$.  Then there is a finite field extension $L/K$ with $u,v \in R \otimes_KL$.  Thus $u$ is a unit of $R \otimes_K L$, so $u^{p^n} \in L^\times \subseteq (K^\alg)^\times$ for some $n \in \N$.  Since $K^\alg$ is a perfect field, we may choose $c \in K^\alg$ such that $c^{p^n} = u^{p^n}$. Then $u-c$ is nilpotent, so (4) holds.

Finally suppose (4) holds.  Let $L/K$ be a field extension.  We have $R \otimes_K L^\alg \cong (R \otimes_K K^\alg) \otimes_{K^\alg} L^\alg$, which by assumption and Proposition~\ref{pr:Sawin} is \ua\ with all units in $L^\alg$ mod nilpotents.  In particular, for any unit $u$, there is some $\lambda \in L^\alg$ and $n \in \N$ with $(u-\lambda)^{p^n} = 0$.  Thus, $u^{p^n} \in L^\alg$, so $u$ is \pins\ over $L^\alg$.  But $R \otimes_K L^\alg \cong (R \otimes_K L) \otimes_L L^\alg$, so by Lemma~\ref{lem:fielddescentcharp}, (1) follows.
\end{proof}

\begin{defn}\label{def:geoua}
If $R$ is a ring and $K$ is a field inside $R$, call $R$ \emph{geometrically \ua\ over $K$} if the equivalent conditions of Theorem~\ref{thm:geom0} or \ref{thm:geomp} hold.
\end{defn}

\begin{example}\label{ex:notgeom}
It follows from the definition that any ring geometrically \ua\ over a field is \ua.  However, the converse does not hold.

Let $R = \R[X,Y] / (X^2 + Y^2 -1)$, the coordinate ring of the unit circle.  Set $S := R \otimes_\R \C = \C[X,Y] / (X^2 + Y^2 -1)= \C[X,Y] / ((X+iY)(X-iY)-1)$.  I claim that $U(S) = \{\lambda \cdot (x+iy)^z \mid \lambda \in \C^\times, z \in \Z\}$.  To see this, first note that any such element is in $S$, as $(x+iy)^{-1} = x-iy$ by construction. Next, note that $\C[t, t^{-1}] \cong \C[U,V]/ (UV-1) \cong \C[X,Y] / (X^2 + Y^2 - 1)$ via $t \mapsto u$, $u \mapsto x+iy$, and $v \mapsto x-iy$, the latter an isomorphism since $X+iY, X-iY$ are a basis for the $\C$-vector space generated by $X$ and $Y$.  Then since $U(\C[t, t^{-1}]) = \{\lambda \cdot t^z \mid \lambda \in \C^\times, z \in \Z\}$, it follows that the group of units of $S$ is as claimed.  But then one sees immediately that $S$ is not \ua, since $x+iy+1$ is neither zero nor in the given set.  Thus, $R$ is not geometrically \ua\ over $\R$.

It remains to show that $R$ itself \emph{is} \ua.  To see this, first consider the following claim.

\noindent \textbf{Claim:} For all positive integers $n$, we have $(x+iy)^n = \phi_n(x) + i \psi_n(x)y$, where $\phi_n, \psi_n \in \R[X]$ are polynomials such that $\deg \phi_n = n$, $\deg \psi_n =n-1$, and $\phi_n$, $\psi_n$ have the same leading coefficient.

\begin{proof}[Proof of claim]
Proceed by induction, where the base case is given by $\phi_1=X$, $\psi_1 = 1$.  So let $n\geq 1$ and assume the claim holds for $n$.  Let $c := $ the common leading coefficient of $\phi_n$ and $\psi_n$. Then \begin{align*}
(x+iy)^{n+1} &= (x+iy)^n (x+iy) = (\phi_n(x) + i\psi_n(x)y)(x+iy)\\
&=(x \phi_n(x) - y^2 \psi_n(x)) + iy(x \psi_n(x) + \phi_n(x)) \\
&= (x \phi_n(x) + (x^2-1)\psi_n(x)) + iy(x \psi_n(x) + \phi_n(x)).
\end{align*}
We have $\deg x \phi_n(x) = \deg ((x^2-1) \psi_n(x)) = n+1$, with both polynomials having leading coefficient $c$, whence $x \phi_n(x) + (x^2-1)\psi_n(x)$ has leading coefficient $2c$ (since $\chr \R \neq 2$).  Similarly, we have $\deg x \psi_n(x) = \deg \phi_n(x) = n$, with both polynomials having leading coefficient $c$, so the leading coefficient of their sum is $2c$.  Hence, with $\phi_{n+1}(X) := X\phi_n(X) + (X^2-1)\psi_n(X)$ and $\psi_{n+1}(X) := X \psi_n(X) + \phi_n(X)$, the claim is proved.
\end{proof}
It remains to show that the only units of $R$ are in $\R$.  Since $U(R) \subseteq U(S)$, it suffices to show that $U(S) \cap R \subseteq \R$.  Accordingly, let $u$ be a unit of $S$ that is in $R$.  Then there is some nonzero complex number $\lambda$ and some $z \in \Z$ with $u=\lambda(x+iy)^z$.  By replacing $u$ with its inverse, we may assume $z\geq 0$.  Suppose $z\geq 1$.  Write $\lambda = a+bi$, with $a,b \in \R$. Then by the claim above, we have \begin{align*}
u &= \lambda(x+iy)^z = (a+bi)(\phi_z(x) + iy\psi_z(x)) \\
&= (a\phi_z(x) - by\psi_z(x)) + i(ay\psi_z(x) + b\phi_z(x)).
\end{align*}
Since $u$ can have no imaginary part, it follows that $ay\psi_z(x) + b \phi_z(x) = 0$.  Since $R \cong \R[x] \oplus y\R[x]$ as vector spaces over $\R$, it further follows that $a\psi_z(x) = b\phi_z(x) = 0$ in $\R[x]$.  But since $\psi_z$ and $\phi_z$ are nonzero polynomials, it follows that $a=b=0$, so that $\lambda = 0$, a contradiction.  Thus, $z=0$, so $u=\lambda = a+bi \in R$, so that $b=0$ and $u=\lambda = a \in \R^\times$.  Hence, $R$ is \ua, but not geometrically \ua\ over $\R$.

It thus follows from Theorem~\ref{thm:FTAiffua} that $S^1 = V_\R(X^2 + Y^2 -1)$ is fundamental, while $V_\C(X^2 + Y^2 -1)$ is not.
\end{example}

\section{Components and minimal primes}

From a geometric standpoint, it is reasonable to ask whether unit-additivity of a ring with connected spectrum is equivalent to unit-additivity of its irreducible components.  I show in this section that the ``$\Leftarrow$'' implication holds, while the ``$\Rightarrow$'' implication fails.  However, I do not know whether the geometric analogue holds even of the ``$\Leftarrow$'' implication for locally \ua\ schemes, even those of finite type over a field.
\begin{prop}\label{pr:modminprimes}
Let $R$ be a ring with connected spectrum and only finitely many minimal primes.  Suppose $R/\p$ is \ua\ for all minimal primes $\p$.  Then $R$ is \ua.
\end{prop}

\begin{proof}
Let $u$ be a unit of $R$.  Let $t=u+1$.  Let $X = \{\p \in \Min R \mid t+\p \text{ is a unit in } R/\p\}$.  Let $Y := \{\p \in \Min R \mid t\in \p \}$.  Since $R/\p$ is a \ua\ domain for each $\p \in \Min R$ and $u+\p$ is a unit in each $R/\p$, it follows that $X \cup Y = \Min R$.

Let $X' := \{P \in \Spec R \mid P $ contains an element of $X\}$, and $Y' := \{P \in \Spec R \mid P $ contains an element of $Y\}$.  Then $X' \cup Y' = \Spec R$, and since $\Min R$ is finite, each of $X'$, $Y'$ is closed in $\Spec R$.  .  

If $X'=\Spec R$, then $t$ is a unit mod every minimal prime.  That is, for each $\p \in \Min R$, there exists $c_\p \in R$ with $1-tc_\p \in \p$.  Let $a := \prod_{\p \in \Min R} (1-tc_\p)$.  Then $a$ is in every minimal prime, hence nilpotent, so there is some $n\in \N$ with $a^n =0$.  But since each $1-tc_\p \equiv 1 \mod tR$, it follows that $0 = a^n \equiv 1 \mod tR$.  That is, $1\in tR$, so $t$ is a unit.

If $Y' = \Spec R$, then $t$ is in every minimal prime of $R$, so $t$ is nilpotent.

Since $\Spec R$ is connected, the remaining case is where  $X' \cap Y' \neq \emptyset$.  Let $P \in X' \cap Y'$.
Then there exist $\p\in X$ and $\q \in Y$ with $\p + \q \subseteq P$.  Since $t+\p$ is a unit in $R/\p$, it follows that it cannot be in any prime ideal of $R/\p$, whence $t \notin P$.  On the other hand, $t \in \q \subseteq P$, which is a contradiction.
\end{proof}

\begin{rem}
The connected spectrum condition above is necessary.  To see this, let 
$S,T$ be unit-additive domains whose fields of units are not both isomorphic to $\F_2$.  Then $R := S\times T$ has exactly two minimal primes $\p, \q$, such that $R/\p \cong S$ and $R/\q \cong T$. However, $R$ is not \ua\ by Proposition~\ref{pr:uaprod}. 
\end{rem}

The following question is then natural.

\begin{question}\label{q:schemecomponents}
    Let $X$ be a connected variety of finite type over a(n algebraically closed) field $k$.  Suppose all irreducible components of $X$ are locally \ua.  Is $X$ locally \ua?
\end{question}

\begin{example}\label{ex:cantdescendcomponents}
The converse to Proposition~\ref{pr:modminprimes} is false.  Let $A = k[X,Y,Z]$, $R = A/(XYZ-Z)$, $\p = (xy-1)R$, $P = (XY-1)A$, $\q = zR$, and $Q = ZA$.  Note that $R/\p \cong A/P \cong k[X, X^{-1},Z]$ and $R/\q \cong A/Q \cong k[X,Y]$ by natural diagrams. Note also that $\p$, $\q$ are the minimal primes of $R$.  Since $R/\p$ is not \ua\ (as $X$ is a unit but $X+1$ is not), it will suffice to show that $R$ is \ua.  I will show that all units of $R$ are in $k$.

Let $f \in R$ be a unit.  Let $F$ be a lift of $f$ to $A$.  Let $\phi: R \onto A/P$ and $\theta: R \onto A/Q$ be the natural maps.
Since $f$ is a unit, we also have that $\phi(f)$ and $\theta(f)$ are units.  But the units of $k[X,X^{-1},Z]$ all look like $cX^n$ for $n \in \Z$ and $c \in k^\times$. Thus, without loss of generality $F=(XY-1)G + c X^n$ for some $G \in A$ and $n \geq 0$ (since if $n<0$, one can switch the roles of $X$ and $Y$).  On the other hand, the units of $k[X,Y]$ are all in $k$.  So $F = ZH + b$, for some $H \in A$ and some $b \in k^\times$.  We have 
$(XY-1)G + c X^n = ZH + b$, which after subtracting become
$cX^n - b = ZH - (XY-1)G \in (XY-1, Z)A$.

Now let $\ia = P+Q$, which is a prime ideal since $A / \ia \cong k[X, X^{-1}]$.  Let $\beta: A \onto k[X, X^{-1}]$ be the corresponding map.  Then $\beta(cX^n -b) = cX^n - b =0\in k[X, X^{-1}]$, whence $b=c$ and $n=0$.  Thus, $ZH = (XY-1)G$, so $G$ must be a multiple of $Z$, say $G=ZG'$.  Then $ZH = (XY-1) ZG'$, so $H = (XY-1)G'$.  Thus, $F = (XY-1)ZG' + c \in c + (XYZ-Z)A$, so that $f = c \in k^\times$.

Thus, it also follows that a locally \ua\ scheme need not have all its irreducible components locally \ua.   That is, the converse to  Question~\ref{q:schemecomponents} has a negative answer.
\end{example}
 
\section*{Acknowledgment}
The author thanks Will Sawin \cite{Saw-MOunits} for providing the proof of the crucial reduced case of Proposition~\ref{pr:Sawin}.  I am also grateful to Sean Lawton for comments that improved the exposition of Section~\ref{sec:fund}.  I am also thankful to the referee for a thoughtful and interesting reading of the paper, and for pointing out places where I could clarify further.

Finally, I extend a warm note of appreciation to Jay Shapiro, who investigated these phenomena with me in the first place and, in the current paper, made minor suggestions, pointed out Example~\ref{ex:nouniquemaximal}, and proved Proposition~\ref{pr:Gdom} and Corollary~\ref{cor:hasGdom}.

\providecommand{\bysame}{\leavevmode\hbox to3em{\hrulefill}\thinspace}
\providecommand{\MR}{\relax\ifhmode\unskip\space\fi MR }
\providecommand{\MRhref}[2]{%
	\href{http://www.ams.org/mathscinet-getitem?mr=#1}{#2}
}
\providecommand{\href}[2]{#2}


\begin{thebibliography}{Kap70}
	
	\bibitem[AM69]{AtMac-ICA}
	Michael~F. Atiyah and Ian~G. Macdonald, \emph{Introduction to commutative
		algebra}, Addison-Wesley Publishing Co., Reading, Mass.-London-Don Mills,
	Ont., 1969.
	
	\bibitem[And]{mse-lcconn}
	Thomas Andrews, \emph{Locally constant functions on connected spaces are
		constant}, Mathematics Stack Exchange,
	URL:https://math.stackexchange.com/q/44863 (version: 2011-06-13).
	
	\bibitem[C{\u a}l15]{Cal-UU}
	Grigore C{\u a}lug{\u a}reanu, \emph{U{U} rings}, Carpathian J. Math.
	\textbf{31} (2015), no.~2, 157--163.
	
	\bibitem[DJ20]{DJ-utriv}
	David~E. Dobbs and No\^{o}men Jarboui, \emph{Associative rings in which 1 is
		the only unit}, Palest. J. Math. \textbf{9} (2020), no.~2, 604--619.
	
	\bibitem[ES25]{nmeSh-unitadd}
	Neil Epstein and Jay Shapiro, \emph{Rings where a non-nilpotent sum of units is
		a unit}, J. Algebra \textbf{672} (2025), 120--144, {\em erratum} {\bf 679}
	(2025), 65--66.
	
	\bibitem[For14]{Fo-unitaffine}
	Timothy~J. Ford, \emph{The group of units on an affine variety}, J. Algebra
	Appl. \textbf{13} (2014), no.~8, 1450065, 27 pages.
	
	\bibitem[FT93]{FrTay-ANT}
	Albrecht Fr{\"o}hlich and Martin~J. Taylor, \emph{Algebraic number theory},
	Cambridge Studies in Advanced Mathematics, vol.~27, Cambridge Univ. Press,
	Cambridge, 1993.
	
	\bibitem[Har77]{Hart-AG}
	Robin Hartshorne, \emph{Algebraic geometry}, Graduate Texts in Mathematics,
	no.~52, Springer-Verlag, New York-Heidelberg, 1977.
	
	\bibitem[HR01]{HeiRo-utriv}
	William Heinzer and Moshe Roitman, \emph{Principal ideal domains and
		{E}uclidean domains having 1 as the only unit}, Comm. Algebra \textbf{29}
	(2001), no.~11, 5197--5208.
	
	\bibitem[Kap70]{Kap-CR}
	Irving Kaplansky, \emph{Commutative rings}, Allyn and Bacon Inc., Boston, 1970.
	
	\bibitem[Per08]{Per-AG}
	Daniel Perrin, \emph{Algebraic geometry. {A}n introduction}, Universitext,
	Springer-Verlag, London, Ltd., London; EDP Sciences, Les Ulis, 2008,
	Translated from the 1995 French original by Catriona Maclean.
	
	\bibitem[Saw]{Saw-MOunits}
	Will Sawin, \emph{Do groups of units change base nicely, assuming the fields
		are algebraically closed?}, MathOverflow,
	URL:https://mathoverflow.net/q/461587 (version: 2024-01-05).
	
\end{thebibliography}
\end{document}